\documentclass[11pt,reqno]{amsart}

\usepackage{amssymb, amsfonts, amsmath, mathrsfs, bbm, dsfont, euscript, stmaryrd, verbatim, setspace, mathtools, lineno, cite, extarrows}
\usepackage{textcomp}
\usepackage{tikz}
\usepackage{tikz-cd}
\usepackage[lite]{amsrefs}
\usepackage[colorlinks, colorlinks, linkcolor=blue, anchorcolor=blue, urlcolor=blue, citecolor=blue]{hyperref}
\allowdisplaybreaks[4]
\usepackage{relsize}
\usepackage{mdframed}
\usepackage{framed}

\newcommand{\fK}{{\mathbb K} }
\newcommand{\fR}{{\mathbb R} }
\newcommand{\fC}{{\mathbb C} }

\newcommand{\id}{ {\mathrm{id} } }

\newcommand{\cF}{ {\mathcal{F}} }

\newcommand{\cD}{ {\mathcal{D}} }

\newcommand{\Ker}{ {\mathrm{Ker}} }

\newcommand{\cH}{ {\mathcal{H}} }

\newcommand{\cK}{ {\mathcal{K}} }

\newcommand{\frakg}{ {\mathfrak{g}} }

\newcommand{\frakl}{ {\mathfrak{l}} }

\newcommand{\fraku}{ {\mathfrak{u}} }

\newcommand{\PBW}{ {\mathrm{PBW}} }

\newcommand{\ad}{ {\mathrm{ad}} }

\newtheorem{definition}{Definition}[section]
\newtheorem{lemma}[definition]{Lemma}
\newtheorem{proposition}[definition]{Proposition}
\newtheorem{corollary}[definition]{Corollary}
\newtheorem{theorem}[definition]{Theorem}
\newtheorem{remark}[definition]{Remark}
\newtheorem{example}[definition]{Example}

\newtheorem{Thm}{Theorem}

\numberwithin{equation}{section}

\begin{document}

\title[]
{From smooth dynamical twists to twistors of quantum groupoids}

\thanks{Research partially supported by
	NSFC grant 12301050 and
Research Fund of Nanchang Hangkong University EA202107232 (Cheng),
NSFC grants 12071241 (Chen) and 12271323 (Qiao),
Key Project of Jiangxi Natural Science Foundation grant 20232ACB201004 (Xiang)}

\author{Jiahao Cheng}
\address{Jiahao Cheng, College of Mathematics and Information Science, Center for Mathematical Sciences, Nanchang Hangkong University, Nanchang, 330063, P.R. China}
\email{jiahaocheng@nchu.edu.cn}

\author{Zhuo Chen*}
\address{Zhuo Chen, Department of Mathematics, Tsinghua University, Beijing, 100084, P.R. China}
\email{chenzhuo@mail.tsinghua.edu.cn}

\author{Yu Qiao}
\address{Yu Qiao, School of Mathematics and Statistics, Shaanxi Normal University, Xi'an, 710119, P.R. China}
\email{yqiao@snnu.edu.cn}

\author{Maosong Xiang}
\address{Maosong Xiang, School of Mathematics and Statistics, Huazhong University of Science and Technology, Wuhan, 430074, P.R. China}
\email{msxiang@hust.edu.cn}

\begin{abstract}
Consider a Lie subalgebra $\frakl \subset \frakg$ and an $\frakl$-invariant open submanifold $V \subset \frakl^{\ast}$. We demonstrate that any smooth dynamical twist on $V$, valued in $U(\frakg) \otimes U(\frakg)\llbracket \hbar \rrbracket$, establishes a twistor on the associated quantum groupoid when combined with the Gutt star product on the cotangent bundle $T^\ast L$ of a Lie group $L$ that integrates $\frakl$. This result provides a framework for constructing equivariant star products from smooth dynamical twists on those Poisson homogeneous spaces  arising from nondegenerate polarized Lie algebras, leveraging the structure of twistors of quantum groupoids.
\end{abstract}

\maketitle

\vspace{-0.4cm}
\begin{small}
\hspace*{0.72cm}
\hspace*{3cm}
quantum dynamical Yang-Baxter equation. \\
\hspace*{1.1cm}
\end{small}

\tableofcontents

\section{Introduction}
One of Drinfeld's significant result asserts that solutions to the quantum Yang-Baxter equation give rise to quantum groups \cites{Drinfeld-QuantumGroup, Drinfeld-HopfYangBaxter}. Similarly, solutions to the quantum \textit{dynamical} Yang-Baxter equation (QDYBE), also known as the Gervais-Neveu-Felder equation \cites{Gervais-Neveu, Felder}, lead to the formation of \textit{dynamical} quantum groups \cites{Etingof-Varchenko-Exchange, Etingof-Varchenko}. We recall a useful method for finding solutions to the QDYBE established in \cite{Babelon-Bernard-Billey}. Let $\frakg$ be a finite-dimensional Lie algebra, and $\eta$ an \textit{abelian} subalgebra of $\frakg$. 
Consider a smooth function $F \in C^{\infty}(V)\llbracket \hbar \rrbracket \otimes U(\frakg) \otimes U(\frakg)$ defined on an open submanifold $V$ of $\eta^{\ast}$, and valued in $U(\frakg) \otimes U(\frakg)\llbracket \hbar \rrbracket$, where $\hbar$ is a formal variable. Assume further that $F$ is of the form $F = 1 \otimes 1 \otimes 1 + O(\hbar)$ and is $\eta$-invariant. The function $\mathsf{R} := F_{21}^{-1} \cdot F$ is a solution to the associated QDYBE over the \textit{abelian} base $\eta^{\ast}$ if and only if $F$ satisfies a shifted cocycle condition. This function $F$ is referred to as a smooth \textbf{dynamical twist}
\cites{Babelon-Bernard-Billey, Etingof-Nikshych-root, Etingof-Varchenko-Exchange}.

 The formation of an \textit{elliptic} quantum group, as introduced by Felder~\cites{Felder-Elliptic, Felder}, originates from a smooth dynamical twist $F$. This quantum group is characterized as a family of quasi-Hopf algebras, denoted by $H_\lambda := (U(\frakg)\llbracket \hbar \rrbracket, \Delta_{\lambda})$, which is parameterized by $\lambda \in \eta^{\ast}$. The coproduct $\Delta_{\lambda}$ is defined as $F^{-1}(\lambda)\Delta F(\lambda)$, resulting in a family of coproducts on $U(\frakg)\llbracket \hbar \rrbracket$ that are \textit{not necessarily coassociative}. An important advancement in this area is presented in~\cite{Xu-R-matrices}, where the algebra $\cH = \cD \otimes U(\frakg)\llbracket \hbar \rrbracket$ is enriched by the algebra $\cD$, the algebra of smooth differential operators on $V$. This enrichment allows $\cH$ to possess a \textbf{quantum groupoid} structure, as defined in~\cite{Xu-QuantumCNRS}. Quantum groupoids, a special class of Hopf algebroids, serve as a unification of quantum groups and star products~\cite{Xu-Quantum}. In this context, the role of $F$ is supplanted by $\cF := F \cdot \Theta_0$, where $\Theta_0$ is expressed as $\exp(\sum_i \hbar \frac{\partial}{\partial \lambda_i} \otimes l_i)$. This $\cF$, referred to as a \textbf{twistor}~\cite{Xu-R-matrices} of $\cH$, transforms $\cH$ into a new quantum groupoid $\cH_{\cF}$, known as the \textbf{dynamical quantum groupoid}. Notably, the coproduct in $\cH_{\cF}$ is \textit{coassociative} and extends the coproduct $\Delta_{\lambda}$ of $H_{\lambda}$ in a natural manner. Another construction of Hopf algebra alike objects associated with solutions of the QDYBE,
 known as the \emph{dynamical quantum groups}, was developed in
 \cites{Etingof-Varchenko, Etingof-Varchenko-Exchange, Etingof-Nikshych-root, Felder-Elliptic}.

These achievements   have been predicated on the assumption that \(\eta \subset \frakg\) is an abelian subalgebra. Therefore, it is reasonable to anticipate the existence of a universal or generalized construction applicable to \textit{non-abelian} subalgebras. This note  provides a comprehensive account of a construction of dynamical quantum groupoids for a non-abelian Lie subalgebra \(\frakl \subset \frakg\).  Within this context, let $V$ be an $\frakl$-invariant open submanifold of $\frakl^{\ast}$. An $\frakl$-invariant element $F$ in $C^{\infty}(V) \llbracket \hbar \rrbracket \otimes U(\frakg) \otimes U(\frakg)$ is termed a \textbf{smooth dynamical twist} if it satisfies specific shifted cocycle conditions, as detailed in the literature~\cites{Xu-Nonabelian, Enriquez-Etingof1, Enriquez-Etingof2}.
The significance of such a smooth dynamical twist lies in its ability to generate a solution $\mathsf{R}:=F_{21}^{-1}\ast_{\PBW} F$ to the  QDYBE. Here the operation $\ast_{\PBW}$ denotes the star product derived from the Poincare-Birkhoff-Witt (PBW) isomorphism.

Classical limits of smooth dynamical twists
are \textit{triangular classical dynamical $r$-matrices}  which satisfy the
classical \textit{dynamical} Yang-Baxter equation (CDYBE) \cite{Xu-Nonabelian}.
A more general shifted-cocycle equation for dynamical twists with an \textbf{associator} term is proposed in
\cite{Enriquez-Etingof1}, we refer to \cites{Enriquez-Etingof-Marshall} on further developments.
The quantization of the \textbf{Alekseev-Meinrenken dynamical $r$-matrices} \cite{Alekseev-Meinrenken},
which are non-triangular and are over the nonabelian base
$\frakg^{\ast}$, is constructed in such setting \cite{Enriquez-Etingof1}.
The classical limits of dynamical twists in \cite{Enriquez-Etingof1}
are solutions of the \textit{modified} CDYBE \cite{LX}.
In this note, we only consider the \textit{trivial} associator case.
For discussions of Hopf algebroids and quantum groupoids in other fields like invariants of knots and $3$-manifold invariants, dg geometry, and shifted Poisson geometry, we refer the reader to \cites{Chen-Xiang-Xu, Kalmykov-Safronov, Nikshych-Turaev-Vainerman}.

Consider a Lie group $L$ that integrates the Lie algebra $\mathfrak{l}$. Gutt~\cite{Gutt} developed a specific star product, now referred to as the \textit{Gutt star product}, on the cotangent bundle $T^{\ast}L \cong \mathfrak{l}^{\ast} \times L$ of $L$.
A normal ordering version of this star product \cite{Xu-Nonabelian}, when
regarded as a bi-differential operator on $V\times L \subset \frakl^{\ast}\times L$, can be represented by an element
 \[
 \Theta_{\mathrm{Gutt}} \in \left( \mathcal{D} \otimes U(\mathfrak{l}) \llbracket \hbar \rrbracket \right) \otimes_{C^{\infty}(V)\llbracket \hbar \rrbracket} \left( \mathcal{D} \otimes U(\mathfrak{l}) \llbracket \hbar \rrbracket \right),
 \]
 which we designate as a \textbf{Gutt twistor}.
 (Here $\cD\otimes U(\frakl)$ is regarded as the space of left $\frakl$-invariant differential operators on
 $V\times L$.)

Here is the main result of this note.
\begin{Thm}[=Theorem~\ref{MAIN}]\label{Thm A}
An $\frakl$-invariant element $F \in C^{\infty}(V)\llbracket \hbar \rrbracket \otimes U(\frakg)\otimes U(\frakg)$ is a smooth dynamical twist if and only if the product $\cF:=\overline{F\ast_{\PBW}}\cdot \Theta_{\mathrm{Gutt}}$ of an element $\overline{F\ast_{\PBW}}$~\eqref{Fstaroverline} and the Gutt twistor $\Theta_{\mathrm{Gutt}}$ is a twistor of the quantum groupoid $\cH$.	
\end{Thm}
Consequently, quantum groupoids $\cH_{\cF}$ twisted by $\cF$ can be viewed as dynamical quantum groupoids over the \textit{nonabelian} base $\frakl^*$. The diagram below illustrates the relationship between our  result and existing works:
\begin{equation*}
\begin{tikzcd}[row sep=1.5 em, column sep=2.3em]
\Bigg\{\!\!\!\!\!
\begin{aligned}
& \text{smooth dynamical twists} \qquad  \\
& F\in C^{\infty}(V)\llbracket \hbar \rrbracket \otimes U(\frakg)\otimes U(\frakg)
\end{aligned} \Bigg\}
\arrow[d, "\text{Alekseev-Calaque \cite{Alekseev-Calaque} }"]
\arrow[r, "\textbf{Theorem}~\ref{Thm A}"] &
\Bigg\{\text{twistors $\cF$ of $\cH=\cD\otimes U(\frakg)\llbracket \hbar \rrbracket$} \Bigg\} \arrow[d, leftrightarrow, "\text{Xu \cite{Xu-QuantumCNRS}}"] \\
\Bigg\{\!\!\!\!
\begin{aligned}
& \text{right $L$-invariant compatible} \quad\quad\\
& \text{star products on}~C^{\infty}(V\times G)\llbracket \hbar \rrbracket	
\end{aligned}	
\Bigg\}
\arrow[u, shift left=5ex, "\text{Xu \cite{Xu-Nonabelian}}"]
\arrow[r, hook]
& \Bigg\{ \text{star products on}~C^{\infty}(V\times G)\llbracket \hbar \rrbracket \Bigg\}\, .	
	\end{tikzcd}
\end{equation*}

In this diagram,  $G$ denotes a Lie group integrating $\frakg$ and $L\subset G$ is a Lie subgroup integrating the inclusion $\frakl \subset \frakg$.  
\begin{itemize}
\item The lower horizontal arrow is the natural inclusion; The upper horizontal arrow is illustrated by our main result.
\item The left bottom-to-top-arrow is due to~\cite{Xu-Nonabelian}*{Corollary 3.6}. 
\item The left top-to-bottom-arrow is pointed out in the proof of  ~\cite{Alekseev-Calaque}*{Proposition 1.3}. The left vertical correspondence is one-to-one. In fact, this correspondence has appeared in the setting $\eta \subset \mathfrak{g}$ being abelian \cite{Xu-R-matrices}.
\item The right vertical correspondence is a bijection proved in \cite{Xu-QuantumCNRS}. It is further shown in \cite{Xu-QuantumCNRS} that star products can be equivalently described in terms of twistors of quantum groupoids.
\end{itemize}

The diagram suggests that our  {Theorem ~\ref{Thm A}} can be interpreted as a natural ``inclusion" from smooth dynamical twists to twistors. Additionally, we note that the abstract construction of bialgebroid twists from dynamical twists, as discussed in~\cite{Donin-Mudrov-2}, could well be related to   our results.

 \medskip

This note is structured to guide the reader through a comprehensive exploration of the topic. In Section \ref{Sec:2}, we   revisit  fundamental concepts related to Hopf algebroids, twistors, quantum groupoids, and dynamical twists.   Section \ref{Sec:3} is dedicated to stating and proving  Theorem~\ref{Thm A}. Finally, in Section \ref{Sec:4}, we apply Theorem~\ref{Thm A} to the Enriquez-Etingof formal dynamical twists, which arise from nondegenerate polarized Lie algebras \cite{Enriquez-Etingof2}. This application offers a novel perspective on equivariant star products on the associated Poisson homogeneous spaces, approached from the standpoint of quantum groupoids.

 \smallskip

\textbf {Conventions and Notations}

\begin{itemize}
\item The base field $\fK$ is either $\fR$ or $\fC$. All vector spaces are over $\fK$.
\item The notation $\fK\llbracket \hbar \rrbracket$ stands for the algebra of formal power series over $\fK$.
\item A $\fK\llbracket \hbar \rrbracket$-module $M$ is called completed if $ M=\varprojlim_k M/ \hbar^{k} M$.
\item The symbol $\cdot\otimes_{R}\cdot$ means the tensor product over some associative $\fK$-algebra $R$.  When $R = \fK$, we simplify it as $\cdot\otimes\cdot$;
    When $R$ is a completed associative $\fK\llbracket \hbar \rrbracket$-algebra,  $\cdot \otimes_{R}\cdot$ means the completed tensor product.
\item When $V$ is a $\fK$-vector space and $M$ is a completed $\fK\llbracket \hbar \rrbracket$-module, we use $V \otimes M$ to denote the completed tensor product $\varprojlim_k V\otimes_{\fK} (M/\hbar^k M)$ (and $M\otimes V$ is similar).
\end{itemize}

\subsection*{Acknowledgements}

We would like to thank Zhangju Liu and Ping Xu
for fruitful discussions and useful comments.

\section{Preliminaries}\label{Sec:2}

\subsection{Hopf algebroids and twistors}\label{Section1}
We start with recalling from \cites{Xu-Quantum, Xu-QuantumCNRS} some basics of Hopf algebroids.
Let us fix an associative $\fK$-algebra $R$ with the unit $1$.
\begin{definition}\label{Def:Hopfalgebroid}
A Hopf algebroid over $R$ consists of a sextuple $(\cH\, , \alpha\, , \beta\, , \cdot\, , \Delta\, , \epsilon)$:
\begin{enumerate}
\item $(\cH, \cdot)$ is an associative $\fK$-algebra with unit $1$.
\item $\alpha\colon R \rightarrow \cH $ is an algebra homomorphism, called
the \textbf{source map} and $\beta\colon R \rightarrow \cH$ is an algebra anti-homomorphism, called the \textbf{target map}, such that their images commute, i.e., $\alpha(a)\beta(b)=\beta(b)\alpha(a)$, for all $a, b\in R$.
\item $\Delta\colon \cH \rightarrow \cH\otimes_{R} \cH$, called the \textbf{coproduct}, is an $(R\,, R)$-bimodule morphism, satisfying $\Delta(1)=1\otimes 1$ and
\begin{equation*}
(\Delta\otimes_{R} \id)\circ \Delta= (\id\otimes_{R} \Delta)\circ \Delta\colon \cH \rightarrow \cH\otimes_{R} \cH\otimes_{R} \cH\, .	
\end{equation*}
\item The product and coproduct are compatible in the sense that
\begin{equation}\label{Important1}
\Delta(x)\cdot( \beta(a)\otimes 1 -1\otimes \alpha(a) )=0\, , \,\,\,\,
\forall x\in \cH, a\in R\, , 	
\end{equation}
\begin{equation*}\mbox{and }~\quad
\Delta(x_1\cdot x_2)=\Delta(x_1)\cdot \Delta(x_2)
\, , \quad \forall x_1\, , x_2\in \cH \, .
\end{equation*}
\item $\epsilon\colon \cH \rightarrow R$, called the \textbf{counit},  is an $(R\,, R)$-bimodule map, satisfying $\epsilon(1)=1$ and
\begin{equation*}
(\epsilon\otimes_{R}\id)\circ	\Delta=(\id\otimes_{R}\epsilon)\circ \Delta=\id.\,
\end{equation*}
Here we have used the identification $R\otimes_{R} \cH \cong \cH\otimes_{R}R \cong \cH$.
\item  {$\Ker(\epsilon)$ is a left ideal of $\cH$.}	
\end{enumerate}
\end{definition}
\begin{example}
Let $A$ be a Lie $\fK$-algebroid over a smooth manifold $M$~\cite{Mackenzie-Book}.
The section space $\Gamma(A)$ of $A$ together with the algebra $C^{\infty}(M)$ of $\fK$-valued smooth functions constitutes a Lie-Rinehart algebra $(C^{\infty}(M)\, , \Gamma(A))$.
The universal enveloping algebra $U(A)$ of $A$ over $M$ is by definition the universal enveloping algebra of the Lie-Rinehart algebra $(C^{\infty}(M)\, , \Gamma(A))$~\cite{Rinehart}, which is indeed a Hopf algebroid over $R=C^{\infty}(M)$ (see \cite{Xu-Quantum}). Both the source and the target maps are the inclusion $C^{\infty}(M) \hookrightarrow U(A)$.
In particular, the universal enveloping algebra $U(\frakg)$ of a Lie $\fK$-algebra $\frakg$ is a Hopf algebroid over $R=\fK$.
\end{example}

Let us abbreviate  $(\cH\, , \alpha\, , \beta\, , \cdot\, , \Delta\, , \epsilon)$ as $\cH$. Here are several facts about this definition:
\begin{itemize}
\item  There is a natural $(R, R)$-bimodule structure on $\cH$ defined by
\[
a \cdot x := \alpha(a)\cdot x, \quad x \cdot b:=\beta(b)\cdot x,
\]
for all $a, b\in R$ and $x\in \cH$, which allows us to define $n$-th power $\cH^{\otimes_{R}^n}$ for all $n$.
\item For two left $\cH$-modules $M_1$ and $M_2$, using Equation \eqref{Important1} we have a left $\cH$-module structure on $M_1\otimes_{R}M_2$ by
\begin{equation*}
x \cdot (m_1\otimes_{R}m_2)=\Delta(x)\cdot (m_1\otimes m_2)\, ,
\end{equation*}
for all $x \in \cH, m_1\in M_1, m_2\in M_2$.
\item  We follow \cite{Lu} to add Condition (6) in our definition, which is however not required in~\cite{Xu-Quantum}. In fact, we need this condition to endow $R$ with a
left $\cH$-module structure $\rhd\colon \cH\times R \rightarrow R $ defined by
\begin{equation*}
x \rhd a := \epsilon(x\cdot \alpha(a) )= \epsilon(x\cdot \beta(a) )\, , 	
\end{equation*}
for all $x\in \cH$ and $a\in R$.
\end{itemize}

Given an element $\cF=\sum_i \cF_{1,i} \otimes_{R} \cF_{2,i} \in \cH\otimes_{R}\cH$,   consider
\begin{align}\label{NewSource}
\alpha_{\cF}\colon &R \rightarrow \cH, & \alpha_{\cF}(a)&=\sum_i \alpha(\cF_{1,i}\rhd a)\cdot \cF_{2,i}, \\
\label{NewTarget}
\beta_{\cF}\colon &R \rightarrow \cH, & \beta_{\cF}(a)&=\sum_i \beta(\cF_{2,i}\rhd a)\cdot \cF_{1,i}, \\
\label{NewProduct}
\ast_{\cF}\colon &R \times R \rightarrow R, & a \ast_{\cF} b &=\sum_i(\cF_{1,i}\rhd a) \cdot (\cF_{2,i}\rhd b), 	
\end{align}
for all $a, b \in R$.

\begin{theorem}[\cites{Xu-R-matrices, Xu-Quantum}] \label{Pre}
Suppose that $\cF\in \cH\otimes_{R}\cH$ satisfies
\begin{equation}\label{TwistorEquation}
(\Delta\otimes_{R} \id)\cF\cdot (\cF\otimes_{R} 1)= (\id\otimes_{R} \Delta)\cF \cdot (1\otimes_{R} \cF)	 \in H\otimes_{R}H\otimes_{R}H\, ,
\end{equation}
and
\begin{equation}\label{counitEquation}
(\epsilon\otimes_{R}\id)	\cF=(\id\otimes_{R}\epsilon)\cF=1\, .	
\end{equation}
Then we have 
\begin{enumerate}
\item $R_{\cF}:=(R\, , \ast_{\cF})$ is an associative $\fK$-algebra with unit $1$.
\item $\alpha_{\cF}\colon R_{\cF}\rightarrow \cH$ is an algebra homomorphism and $\beta_{\cF}\colon R_{\cF}\rightarrow \cH$ is an algebra anti-homomorphism, satisfying $\alpha_{\cF}(a)\cdot \beta_{\cF}(b) = \beta_{\cF}(b) \cdot \alpha_{\cF}(a)$ and
\begin{equation*}
\cF \cdot (\beta_{\cF}(a)\otimes 1 - 1 \otimes \alpha_{\cF}(a))=0, 	
\end{equation*}
for all $a, b \in R$.
\end{enumerate}
\end{theorem}

As a consequence, such an element $\cF$ determines a linear map of $\fK$-vector spaces
\begin{align}\label{Eqt:Fsharp}
\cF^{\sharp}\colon M_1\otimes_{R_{\cF}}M_2 & \rightarrow M_1 \otimes_{R} M_2 \notag \\
\cF^{\sharp}(m_1 \otimes_{R_{\cF}} m_2 ) &:= \sum_i (\cF_{1,i}\cdot m_1) \otimes_{R} (\cF_{2,i}\cdot m_2),
\end{align}
for any two left $\cH$-modules $M_1$ and $M_2$.

\begin{definition}[\cites{Xu-R-matrices, Xu-Quantum}]
The element  $\cF\in \cH\otimes_{R}\cH$ is called a \textbf{twistor} of the Hopf algebroid $\cH$, if it satisfies Equations \eqref{TwistorEquation} and \eqref{counitEquation}, and the linear map $\cF^{\sharp}$ in \eqref{Eqt:Fsharp}  is \textbf{invertible}.
\end{definition}

\subsection{Quantum groupoids and star products}
Given a smooth manifold $M$, let $R$ be a completed associative $\fK\llbracket \hbar \rrbracket$-algebra with unit such that $R/\hbar R = C^{\infty}(M)$.

\begin{definition}[\cites{Xu-Quantum, Xu-QuantumCNRS}]
A quantum universal enveloping algebroid (or QUE algebroid), also called a quantum groupoid, is a topological Hopf algebroid $(\cH\, , \alpha\, , \beta\, , \cdot_{\hbar}\, , \Delta\, , \epsilon)$ over $R$
(i.e., $\cH$ is a completed $\fK\llbracket \hbar \rrbracket$-module, all maps $\alpha$, $\beta$, $\cdot$, $\Delta$, $\epsilon$ are continuous with respect to the $\hbar$-adic topology), such that
the induced Hopf algebroid $\cH/\hbar \cH$ over $C^{\infty}(M)$ is isomorphic to the standard Hopf algebroid $U(A)$ of a Lie algebroid $A$ over $M$.
\end{definition}

\begin{example}\label{Ex:cF-star}
The Hopf algebroid structure on $U(TM)$ over $C^{\infty}(M)$ extends naturally to a quantum groupoid structure on $\cH :=U(TM)\llbracket \hbar \rrbracket$ over $R := C^{\infty}(M)\llbracket \hbar \rrbracket$ .
Consider a formal power series in $\hbar$
\begin{equation*}
\cF=1\otimes_{R}1+ \sum_{k=1} \hbar^{k} B_{k} \in U(TM)\otimes_{C^{\infty}(M)}U(TM)\llbracket \hbar \rrbracket = U(TM)\llbracket \hbar \rrbracket \otimes_{C^{\infty}(M)\llbracket \hbar \rrbracket} U(TM)\llbracket \hbar \rrbracket,
\end{equation*}
where the coefficients $B_k$ are bi-differential operators on $M$.
Then $\cF$ is a twistor of $\cH$ if and only if the product $\ast_{\cF}$ defined by
\begin{equation*}
f \ast_{\cF} g:=\cF(f\, , g),
\end{equation*}
for all $f, g\in C^{\infty}(M)\llbracket \hbar \rrbracket$, is a \textbf{star product} on $M$.

The twistor $\cF$ determines a new quantum groupoid $\cH_{\cF}=(\cH, \alpha_{\cF}, \beta_{\cF}, \cdot, \Delta_{\cF}, \epsilon)$, where the source and target maps $\alpha_{\cF}, \beta_{\cF} \colon R_{\cF}\rightarrow \cH$ are given by the left and right multiplication of the star product $\ast_{\cF}$, respectively, i.e.,
\begin{align*}
\alpha_{\cF}(f)(g) &= f\ast_{\cF}g, &  \beta_{\cF}(f)(g) &= g\ast_{\cF}f,
\end{align*}
for all $f, g \in C^{\infty}(M)$. The coproduct $\Delta_{\cF}$ is given by
\begin{equation*}
\Delta_{\cF}(x):=(\cF^{\sharp})^{-1}( \Delta(x)\cdot \cF ),	
\end{equation*}
for all $x \in \cH$.
\end{example}

\medskip

Let $(\frakl,[-,-]_{\frakl})$ be an $N$-dimensional $\fK$-Lie algebra.
Denote by $\frakl_{\hbar}:=\frakl\llbracket \hbar \rrbracket =\frakl\otimes \fK\llbracket \hbar \rrbracket$ the Lie algebra whose Lie bracket is $\fK\llbracket \hbar \rrbracket$-linear and satisfies $[X\, , Y]_{\frakl_{\hbar}} = \hbar [X\, , Y]_{\frakl}$,  for all $X\, , Y\in \frakl$.
Consider the universal enveloping algebra $U(\frakl_{\hbar})$ of $\frakl_{\hbar}$ over $\fK\llbracket \hbar \rrbracket$. Note that
\[
U(\frakl_{h})/ \hbar U(\frakl_{h}) \cong S(\frakl) \cong \mathrm{Pol}(\frakl^{\ast}),
\]
where $\mathrm{Pol}(\frakl^{\ast})$ denotes the space of polynomial functions on $\frakl^{\ast}$. Thus, one can regard $U(\frakl_{\hbar})$ as a deformation quantization of $\mathrm{Pol}(\frakl^{\ast})$.

Choose a basis $\{l_{1},\cdots,l_{N}\}$ of $\frakl$.  Denote by $\{\lambda_{1},\cdots, \lambda_N\}$  the corresponding coordinate functions on $\frakl^{\ast}$. Then the PBW map is of the form
\begin{align}\label{PBWmap}
\mathrm{Pol}(\frakl^{\ast})\llbracket \hbar\rrbracket & \rightarrow U(\frakl_{\hbar}) , \notag \\
\lambda_{i_1}\cdots \lambda_{i_k} & \mapsto
\frac{1}{k!}\sum_{\sigma\in \mathrm{S}_k}
  l_{i_{\sigma(1)}} \cdots   l_{i_{\sigma(k)}},
\end{align}
where $\mathrm{S}_k$ is the group of permutations of $k$ elements. It is indeed a vector space isomorphism. Via the PBW map, the associative multiplication in $U_{\hbar}(\frakl)$ pulls back to an associative (but not commutative) multiplication $\ast_{\PBW}$ on $\mathrm{Pol}(\frakl^{\ast})\llbracket \hbar \rrbracket$, which is extended naturally to an associative multiplication, still denoted by $\ast_{\PBW}$, on $C^{\infty}(\frakl^{\ast})\llbracket \hbar \rrbracket$, called the \textbf{PBW star product}.
In terms of powers of $\hbar$, we write
\begin{equation*}
f\ast_{\PBW}g=fg+ \frac{1}{2}\hbar \{f\, , g\}_{\pi_{\frakl^{\ast}}}+
\sum_{k\geqslant 2}\hbar^{k} B_k (f\, , g), 	
\end{equation*}
for all $f, g\in C^{\infty}(\frakl^{\ast})$, where $\{- , -\}_{\pi_{\frakl^{\ast}}}$ is the Poisson bracket induced by the linear Poisson structure  $\pi_{\frakl^{\ast}}$ on $\frakl^{\ast}$,  and $B_{k}$'s are a family of bi-differential operators on $\frakl^{\ast}$ (see \cites{Dito, Kathotia, Gutt}).

Denote by $U(T\frakl^{\ast})$ the space of differential operators on $\frakl^{\ast}$.
Note that left multiplications of the PBW star product by smooth functions on $\frakl^{\ast}$ are elements in $U(T\frakl^{\ast})\llbracket \hbar \rrbracket$, satisfying
\begin{equation*}
(f\ast_{\PBW}) \cdot (g\ast_{\PBW}) = (f\ast_{\PBW}g)\ast_{\PBW}
\in U(T\frakl^{\ast})\llbracket \hbar \rrbracket,
\end{equation*}
for all $f, g \in C^{\infty}(\frakl^{\ast})$. Here $\cdot$ is the product in $U(T\frakl^{\ast})$.
By the construction of the PBW map \eqref{PBWmap}, we have
\begin{equation}\label{MonoidalPBW}
(\lambda_{i_1}\cdots \lambda_{i_k})\ast_{\PBW} = \frac{1}{k!}\sum_{\sigma}
  (\lambda_{i_{\sigma(1)}}\ast_{\PBW})\cdot (\lambda_{i_{\sigma(2)}}\ast_{\PBW}) \cdot  \,\, \cdots \,\,
\cdot(\lambda_{i_{\sigma(k)}}\ast_{\PBW}),
\end{equation}
for all monomials of the form $\lambda_{i_1}\cdots \lambda_{i_k}$.

Let $V \subset \frakl^{\ast}$ be an open submanifold which is invariant under the coadjoint action. The PBW star product reduces to an associative multiplication on $C^{\infty}(V)\llbracket \hbar \rrbracket$, still denoted by $\ast_{\PBW}$.
\bigskip

Let $\frakg$ be a Lie algebra such that $\frakl\subset \frakg$ is a Lie subalgebra.
Let $G$ be a Lie group integrating $\frakg$ and $L \subset G$ be a Lie subgroup integrating $\frakl \subset \frakg$.
For any $l \in \frakl \subset \frakg$, denote by $\overset{\rightarrow}{l}$ the left invariant vector field on $L$ or on $G$ generated by $l$.

\begin{definition}[\cite{Xu-Nonabelian}]
 A \textbf{compatible star product} on $V\times G$ is an associative product $\ast$ on $ C^{\infty}(V \times G)\llbracket \hbar \rrbracket \cong  C^{\infty}(V)\llbracket \hbar \rrbracket \otimes  C^{\infty}(G)\llbracket \hbar \rrbracket$, satisfying the following properties:
\begin{enumerate}
\item The constant function $1\otimes 1$ is a unit.
\item It extends the PBW star product on $C^{\infty}(V)\llbracket \hbar \rrbracket$, i.e.,
\begin{equation*}
g_1 \ast g_2= g_1\ast_{\PBW} g_2\, ,
\end{equation*}
for all $g_1\, , g_2\in C^{\infty}(V)$.
\item  For all $f \in C^{\infty}(G)$ and $g \in C^{\infty}(V)$, we have
\begin{equation*}
f \ast g =f\cdot g,
\end{equation*}
 and
\begin{equation*}
g \ast f = \sum_{k=0}^{\infty} \sum_{i_1, \cdots, i_k\in \{1,\cdots,N\}} \frac{\hbar^{k}}{k!} \frac{\partial^{k} g}{\partial \lambda_{i_1} \partial \lambda_{i_2} \cdots \partial \lambda_{i_k}} \cdot \overset{\longrightarrow}{l_{i_1}} \overset{\longrightarrow}{l_{i_2}} \cdots \overset{\longrightarrow}{l_{i_k}}f \, .	
\end{equation*}
\item For all $f_1\, , f_2\in C^{\infty}(G)$, we have
\begin{equation*}
f_1 \ast f_2 = \overset{\xrightarrow{ }}{F }(f_1\, , f_2)\, ,
\end{equation*}
for some $F = 1\otimes 1\otimes 1 + O(\hbar) \in C^{\infty}(V)\llbracket \hbar \rrbracket\otimes U(\frakg)\otimes U(\frakg)$.
 \end{enumerate}
\end{definition}
As a result, a compatible star product $\ast$ is determined by an element
\begin{equation}\label{Eq: F}
F =\sum_{p=0}^{\infty}\hbar^p F_p = \sum_{p=0}^\infty \hbar^p \sum_{r,s} b_{p,rs}\otimes  U_{p,r}\otimes U_{p,s} \in C^{\infty}(V)\llbracket \hbar \rrbracket \otimes U(\frakg)\otimes U(\frakg),
\end{equation}
where $b_{p,rs} \in C^{\infty}(V)$ and $U_{p,r}, U_{p,s}\in U(\frakg)$, satisfying
$F_0 = \sum_{r,s} b_{0,rs}\otimes  U_{0,r}\otimes U_{0,s} = 1\otimes 1 \otimes 1$.

According to~\cite{Xu-Nonabelian}*{Corollary 3.6}, if   $\ast$ is a compatible star product
on $V\times G$ arising from $F$, then $F$ satisfies
\begin{equation}\label{counit2}
(\id\otimes \epsilon\otimes \id)F=(\id\otimes \id\otimes \epsilon)F=1\otimes 1\, ,	
\end{equation}
and
\begin{equation}\label{Nonabelian-Twist2}
(\id\otimes \Delta\otimes \id)F \ast_{\PBW} F_{12}^{(3)} =(\id\otimes \id \otimes \Delta)F  \ast_{\PBW} F_{23},
\end{equation}
where
\begin{align*}
F_{12}^{(3)} & := \sum_{p,r,s} \sum_{k=0}^{\infty}\sum_{i_1, \cdots, i_k}	
\frac{\hbar^{k+p}}{k!} \frac{\partial^k b_{p,rs}}{\partial \lambda_{i_1}\cdots \partial \lambda_{i_k}}
\otimes U_{p,r}\otimes U_{p,s} \otimes (l_{i_1}\cdots l_{i_k}) \\
& \qquad\qquad\qquad\qquad\qquad\qquad\qquad \in
C^{\infty}(V)\llbracket \hbar \rrbracket \otimes U(\frakg)\otimes U(\frakg)\otimes U(\frakl), \\
	F_{23} & := \sum_{p,r,s} \hbar^p b_{p,rs} \otimes 1 \otimes U_{p,r} \otimes U_{p,s}
	 \quad  \in C^{\infty}(V)\llbracket \hbar \rrbracket \otimes  U(\frakl) \otimes U(\frakg) \otimes U(\frakg).
\end{align*}

Note that there is a canonical right $L$-action on $V\times G$ defined by
\begin{equation*}
(\lambda\, , x) \cdot	 y:= (\mathrm{Ad}_{y^{-1}}^{\ast}\lambda\, , xy)\, ,
\end{equation*}
for all $(\lambda\, , x)\in V\times G$ and $y \in L$.
It is proved in \cite{Xu-Nonabelian} that a compatible star product $\ast$ determined by $F \in C^{\infty}(V)\llbracket \hbar \rrbracket\otimes U(\frakg)\otimes U(\frakg)$ is right $L$-invariant, i.e.,
\begin{equation*}
(f\ast g)\bigl( (\lambda\, , x)\cdot y\bigr) = f\bigl( (\lambda\, , x)\cdot y  \bigr) \ast
g\bigl( (\lambda\, , x)\cdot y  \bigr), 	\quad \forall (\lambda\, , x)\in V\times G , y\in L ,
\end{equation*}
if and only if $F$ is $\frakl$-invariant, i.e.,
\begin{equation*}
	 \ad^*_l \otimes 1\otimes 1 (F)  +[1\otimes l\otimes 1+ 1\otimes 1\otimes l\, , F ]=0\, , 	\quad
\forall  l \in \frakl.
\end{equation*}

An $F$ subject to the equations~\eqref{counit2} and~\eqref{Nonabelian-Twist2} is called a  \textbf{smooth dynamical twist} if it is also $\frakl$-invariant~\cites{Enriquez-Etingof2, Enriquez-Etingof1}. Thus, a right $L$-invariant  compatible star product on $V\times G$ determines a smooth dynamical twist. For the converse, see Theorem \ref{MAIN} and the explanation on its relations with some existing works in the next section.

\section{From dynamical twists to twistors of quantum groupoids}\label{Sec:3}
We start by introducing a quantum groupoid arising from a Lie algebra pair $(\frakg, \frakl)$.
Let $G$ be a Lie group integrating $\frakg$ and $L\subset G$ be a Lie subgroup integrating $\frakl \subset \frakg$.
For an $L$-invariant open submanifold $V \subset \frakl^{\ast}$, the product $V\times G$ admits a natural left $G$-action defined by
\begin{equation*}
g\cdot(\lambda\, , x)=(\lambda\, , gx)\, , 	
\end{equation*}
for all $g\in G$ and $(\lambda\, , x)\in V\times G$.
Consider the space of formal power series in $\hbar$ with coefficients being left $G$-invariant
differential operators on $V\times G$, i.e.,
\begin{equation*}
\cH:=\cD\otimes U(\frakg)\llbracket \hbar \rrbracket,
\end{equation*}
where $\cD$ is the algebra of smooth differential operators on $V$.
It is clear that $\cH$ admits a \textit{quantum groupoid} structure over $R:=C^{\infty}(V)\llbracket \hbar\rrbracket$, whose source and target maps,  product, coproduct, and counit are all naturally induced from the two Hopf algebroids $\cD$ and $U(\frakg)$.

\medskip

In~\cite{Gutt}, Gutt constructed an explicit star product on the cotangent bundle $T^\ast L \cong \frakl^{\ast}\times L$,
which realizes deformation quantization of the Poisson structure
\begin{equation*}
\pi=\pi_{\frakl^{\ast}}+\sum_i \frac{\partial}{\partial \lambda_{i}}	\wedge \overset{\rightarrow}{l_{i}}
\end{equation*}
coming from the canonical symplectic structure on the cotangent bundle. A normal ordering version of this star product, which we also call \textbf{Gutt star product}, is given in \cite{Xu-Nonabelian}
\begin{equation}\label{GuttStarProduct}
f \ast_{\mathrm{Gutt}} g := \sum_{k=0}^{\infty} \sum_{i_1, \cdots, i_k} \frac{\hbar^{k}}{k!}
\frac{\partial^k f}{\partial \lambda_{i_1} \cdots \partial \lambda_{i_k}} \ast_{\PBW}
\overset{\longrightarrow}{l_{i_1}} \overset{\longrightarrow}{l_{i_2}} \cdots \overset{\longrightarrow}{l_{i_k}}g\, ,
\end{equation}
for all $f, g\in C^{\infty}(\frakl^{\ast}\times L)\llbracket \hbar \rrbracket$.
Moreover, it is a compatible star product on $\frakl^{\ast}\times L$~\cite{Xu-Nonabelian}.
When restricted to the $L$-invariant open submanifold $V\subset \frakl^{\ast}$, this Gutt star product gives rise to a twistor (cf. Example \ref{Ex:cF-star}):
\begin{equation*}
\Theta_{\mathrm{Gutt}}\in \left( \cD\otimes U(\frakl) \llbracket \hbar \rrbracket
\right) \otimes_{C^{\infty}(V)\llbracket \hbar \rrbracket} \left( \cD\otimes U(\frakl)
\llbracket \hbar \rrbracket \right)
\end{equation*}
of the quantum groupoid $\cD\otimes U(\frakl) \llbracket \hbar \rrbracket $ over $C^{\infty}(V)\llbracket \hbar \rrbracket$.
We  call $\Theta_{\mathrm{Gutt}}$ the \textbf{Gutt twistor}.

It is clear that $\cD\otimes U(\frakl)\llbracket \hbar \rrbracket \subset \cH$ is closed under the product and coproduct. Thus, the Gutt twistor $\Theta_{\mathrm{Gutt}}\in \cH\otimes_{R}\cH $ is also a twistor of $\cH$.

\subsection{Main result}
For each $F\in C^\infty(V) \llbracket \hbar \rrbracket \otimes U(\frakg)\otimes U(\frakg)$ as in~\eqref{Eq: F}, the left multiplication of the PBW star product on $C^\infty(V)\llbracket \hbar \rrbracket $  determines  {an element}
\begin{equation*}
F\ast_{\PBW}:=\sum_{p,r,s} \hbar^p (b_{p,rs}\ast_{\PBW})\otimes U_{p,r}\otimes U_{p,s}
\in \cD \llbracket \hbar \rrbracket \otimes U(\frakg)\otimes U(\frakg). 	
\end{equation*}
Composing with the coproduct $\Delta$ on $\cD \llbracket \hbar  \rrbracket$, we obtain
\begin{align}\label{Fstaroverline}
\overline{F\ast_{\PBW}} &:= \sum_{p,r,s} \hbar^p \Delta(b_{p,rs}\ast_{\PBW}) \otimes (U_{p,r}\otimes U_{p,s}) \notag \\
&\in \left( \cD \llbracket \hbar \rrbracket \otimes_{R} \cD \llbracket \hbar  \rrbracket \right)
\otimes \left( U(\frakg)\otimes U(\frakg) \right) \cong \cH \otimes_{R} \cH.	
\end{align}

\begin{theorem}\label{MAIN}
Let $(\frakg, \frakl)$ be a Lie algebra pair, and $V$ be an $\frakl$-invariant open submanifold in
$\frakl^{\ast}$. Suppose that $F$ is an $\frakl$-invariant element in $C^{\infty}(V) \llbracket \hbar \rrbracket \otimes U(\frakg)\otimes U(\frakg)$.
The following statements are equivalent:
\begin{enumerate}
\item $F$ is a smooth dynamical twist.
\item The element
\begin{equation*}
\cF= \overline{F\ast_{\PBW}} \cdot \Theta_{\mathrm{Gutt}}\in \cH\otimes_{R} \cH
\end{equation*}
is a twistor of the quantum groupoid $\cH:=\cD\otimes U(\frakg)\llbracket \hbar \rrbracket$.
\end{enumerate}
\end{theorem}
We postpone the proof in the next subsection.
Applying Theorem~\ref{MAIN} and \cite{Xu-Quantum}*{Theorem 4.14} to the quantum groupoid $\mathcal{H} = (\cH , \alpha, \beta, \cdot, \Delta, \epsilon)$ over $R$, we obtain the following:
\begin{corollary}\label{MAIN2}
Suppose that $F$ is a smooth dynamical twist. Then the sextuple $\mathcal{H}_{\cF}=(\cH, \alpha_{\cF}, \beta_{\cF}, \cdot, \Delta_{\cF}, \epsilon)$ is a quantum groupoid over $R_{\cF}=(R, \ast_{\cF})$, where the product $\ast_{\cF}$ is given by Equation \eqref{NewProduct},
the source and target maps $\alpha_{\cF}$ and $\beta_{\cF}$ are given by Equation \eqref{NewSource} and \eqref{NewTarget}, respectively, the product $\cdot$ is the usual associative multiplication
and the coproduct $\Delta_{\cF}$ is given by
\begin{equation*}
\Delta_{\cF}(x):=(\cF^{\sharp})^{-1}( \Delta(x)\cdot \cF ),	
\end{equation*}
for all $x \in \cH$.
The resulting quantum groupoid $\cH_{\cF}$ is called the dynamical quantum groupoid.
\end{corollary}
\medskip

We now explain relations between our Theorem \ref{MAIN} and some existing results:
\begin{enumerate}
\item When the Lie subalgebra $\frakl$ is abelian, the Gutt twistor $\Theta_{\mathrm{Gutt}}= \Theta_0 = \exp( \sum_i \hbar \frac{\partial}{\partial \lambda_i}\otimes l_i)$ and the PBW star product $\ast_{\PBW}$ coincide with the usual multiplication of functions. Applying Theorem \ref{MAIN} and Corollary \ref{MAIN2} to this case, we recover the construction of twistors and dynamical quantum groupoids from smooth dynamical twists in~\cite{Xu-R-matrices}.
\item  Note that the space of right $L$-invariant compatible star products on $C^{\infty}(V\times G)\llbracket \hbar \rrbracket$ is isomorphic to that of twistors of the form $\cF=\overline{F\ast_{\PBW}} \cdot \Theta_{\mathrm{Gutt}}$. As a consequence of Theorem \ref{MAIN}, we recover the one-to-one correspondence between smooth dynamical twists and right $L$-invariant compatible star products on $V \times G$~\cites{Alekseev-Calaque, Xu-Nonabelian}.

\item Suppose that $F= \sum_{p=0}^{\infty}\sum_{r,s} \hbar^p b_{p,rs}\otimes  U_{p,r}\otimes U_{p,s}\in C^{\infty}(V) \llbracket \hbar \rrbracket\otimes U(\frakg)\otimes U(\frakg) $ is a smooth dynamical twist.
    Then its classical limit $\theta = \sum_{r,s} b_{1,rs}\otimes  (U_{1,r}\otimes U_{1,s} - \otimes  U_{1,s}\otimes U_{1,r})$ is a triangular classical dynamical $r$-matrice which
    satisfies the classical dynamical Yang-Baxter equation (CDYBE)\cite{Xu-Nonabelian}.
    On the other hand, consider the associated twistor
    \[
    \cF=\overline{F\ast_{\PBW}} \cdot \Theta_{\mathrm{Gutt}} = \sum_i \cF_{1,i}\otimes_{R} \cF_{2,i}\in \cH\otimes_{R}\cH,
    \]
whose classical limit
\begin{equation*}
\lim_{\hbar\rightarrow 0}\frac{1}{\hbar}\sum_i (\cF_{1,i}\otimes_{R} \cF_{2,i} - \cF_{2,i}\otimes_{R} \cF_{1,i} ) = \pi_{\frakl^{\ast}}+\sum_i \frac{\partial}{\partial \lambda_{i}}	\wedge \overset{\rightarrow}{l_{i}} + \overset{\rightarrow}{\theta}
\end{equation*}
is indeed a Poisson bivector on the manifold $V\times G$. Thus, by taking the classical limit of Theorem \ref{MAIN}, we recover a result in \cite{Xu-Nonabelian}, which asserts that an $\frakl$-invariant element $\theta \in C^{\infty}(V)\otimes \wedge^2 \frakg$ satisfies the CDYBE if and only if $\pi_{\frakl^{\ast}}+\sum_i \frac{\partial}{\partial \lambda_{i}} \wedge \overset{\rightarrow}{l_{i}}
+ \overset{\rightarrow}{\theta}$ is a Poisson bivector on $V\times G$.
\end{enumerate}

\subsection{More facts about the Gutt twistor}
 \begin{lemma}\label{Lemma2}
We have
\begin{equation*}
(\Delta\otimes_{R} \id)\Theta_{\mathrm{Gutt}}\cdot (\overline{F\ast_{\PBW}})_{12} =
( \overline{\overline{F\ast_{\PBW}}} )_{12}^{(3)} \cdot (\Delta\otimes_{R} \id)\Theta_{\mathrm{Gutt}} \in \cH\otimes_{R} \cH\otimes_{R}\cH,	
\end{equation*}
where
\begin{equation*}
(\overline{F\ast_{\PBW}})_{12}:=\overline{F\ast_{\PBW}} \otimes 1 \in \cH\otimes_{R} \cH\otimes_{R}\cH, 	
\end{equation*}
and $\overline{\overline{F \ast_{\PBW}}})^{(3)}_{12}  \in  \cH\otimes_{R} \cH\otimes_{R}\cH$ is defined by
\begin{align*}
\sum_{p,r,s,k}\sum_{1 \leq i_1,\cdots, i_k \leq N} & \frac{\hbar^{k+p}}{k!}(\id\otimes_{R}\Delta) \Delta\left(\frac{\partial^k b_{p,rs}}{\partial \lambda_{i_1}\cdots \partial \lambda_{i_k}}\ast_{\PBW}\right) \cdot (U_{p,r}\otimes U_{p,s} \otimes (l_{i_1}\cdots l_{i_k})).
\end{align*}
\end{lemma}
\begin{proof}
The PBW star product $\ast_{\PBW}$ on $C^{ \infty}(V)\llbracket \hbar \rrbracket$ is determined by  an element, called PBW twistor,
\begin{equation*}
\Theta_{\PBW}=\sum_i (\Theta_{\PBW})_{1,i}\otimes_{R}  (\Theta_{\PBW})_{2,i} \in \cD\llbracket \hbar \rrbracket \otimes_{C^{\infty}(V)\llbracket \hbar \rrbracket}
\cD\llbracket \hbar \rrbracket.
\end{equation*}
By Equation \eqref{GuttStarProduct}, we have
\begin{align*}
f\ast_{\mathrm{Gutt}}g & =  \Theta_{\mathrm{Gutt}}(f\, , g) \\
& = \sum_i\sum_k\sum_{i_1, \cdots, i_k} \frac{\hbar^{k}}{k!}
\biggl((\Theta_{\PBW})_{1,i}\frac{\partial^k}{\partial \lambda_{i_1}\cdots \partial\lambda_{i_k}}
\biggr)\otimes_{R}\biggl( (\Theta_{\PBW})_{2,i} \overset{\rightarrow}{l_{i_1}}\cdots
\overset{\rightarrow}{l_{i_k}} \biggr) (f\, , g),	
\end{align*}
for all $f , g\in C^{\infty}(V)$. Thus, the Gutt twistor has the form
\[
\Theta_{\mathrm{Gutt}}=\sum_{q\in I}(\Theta_{\mathrm{Gutt}})_{1,q}\otimes_{R}  (\Theta_{\mathrm{Gutt}})_{2,q}
\]
for some countable index set $I$, where $(\Theta_{\mathrm{Gutt}})_{1,q} \in \cD\llbracket \hbar \rrbracket, (\Theta_{\mathrm{Gutt}})_{2,q} \in \cH$.
Since the Gutt star product is associative, it follows that the Gutt twistor satisfies
 \begin{align*}
& \sum_{q} \bigl(\Delta ((\Theta_{\mathrm{Gutt}})_{1,q}\cdot(b_{p,rs}\ast_{\PBW})\bigr)\otimes_{R} (\Theta_{\mathrm{Gutt}})_{2,q}=\sum_{q}\sum_{k}\sum_{i_1, \cdots, i_k} \\
&\qquad\frac{\hbar^{k}}{k!}\biggl((\id\otimes \Delta)\Delta( \frac{\partial^k b_{p,rs}}{\partial \lambda_{i_1} \cdots \partial \lambda_{i_k}}\ast_{\PBW})\biggl)\cdot\biggl(
\Delta (\Theta_{\mathrm{Gutt}})_{1,q}\otimes(l_{i_1}\cdots l_{i_k})(\Theta_{\mathrm{Gutt}})_{2,q}\biggr).
\end{align*}	
Thus, we have
\begin{align*}
&(\Delta\otimes_{R} \id)\Theta_{\mathrm{Gutt}}\cdot (\overline{F\ast_{\PBW}})_{12}\\
=&\sum_{p,r,s}\sum_{q}\hbar^p\biggl(\bigl(\Delta ((\Theta_{\mathrm{Gutt}})_{1,q}\cdot
(b_{p,rs}\ast_{\PBW})\bigr)\otimes_{R} (\Theta_{\mathrm{Gutt}})_{2,q}\biggr)\cdot
\biggl(U_{p,r}\otimes U_{p,s} \otimes 1\biggr)  \\
=&\sum_{p,r,s}\sum_q \sum_{k}\sum_{i_1,\cdots,i_k} \frac{\hbar^{k+p}}{k!}\biggl((\id\otimes \Delta)\Delta\left( \frac{\partial^k b_{p,rs}}{\partial \lambda_{i_1}\cdots \partial \lambda_{i_k}}\ast_{\PBW}\right)\biggl) \\
&\qquad\cdot\biggl(\Delta (\Theta_{\mathrm{Gutt}})_{1,q} \otimes(l_{i_1}\cdots l_{i_k})(\Theta_{\mathrm{Gutt}})_{2,q}\biggr)\! \cdot \!\biggl(U_{p,r}\otimes U_{p,s} \otimes 1\biggr) \\
= &\sum_{p,r,s}\sum_q \sum_{k}\sum_{i_1,\cdots,i_k} \frac{\hbar^{k+p}}{k!}\biggl((\id\otimes \Delta)\Delta\left( \frac{\partial^k b_{p,rs}}{\partial \lambda_{i_1}\cdots \partial \lambda_{i_k}}\ast_{\PBW}\right) \biggl) \\
&\qquad \cdot \biggl((U_{p,r}\otimes U_{p,s})\cdot\Delta (\Theta_{\mathrm{Gutt}})_{1,q}\biggr)
\otimes\biggl((l_{i_1}\cdots l_{i_k})(\Theta_{\mathrm{Gutt}})_{2,q}\biggr)\\
=& ( \overline{\overline{F\ast_{\PBW}} } )_{12}^{(3)}\cdot(\Delta\otimes_{R} \id)\Theta_{\mathrm{Gutt}}.
\end{align*}
\end{proof}

\begin{lemma}\label{Exchange}
For any polynomial function $g \in \mathrm{Pol}(\frakl^{\ast}) \subset C^{\infty}(V)$
and $\frakl$-invariant element $F \in C^{\infty}(V)\llbracket \hbar \rrbracket \otimes U(\frakg)\otimes U(\frakg)$, we have
\begin{equation*}
\Delta(g\ast_{\mathrm{Gutt}})\cdot \overline{F\ast_{\PBW}} = \overline{F\ast_{\PBW}} \cdot
\Delta(g\ast_{\mathrm{Gutt}}).
\end{equation*}
\end{lemma}
\begin{proof}
For any polynomial function $g\in \mathrm{Pol}(\frakl^{\ast})$, we claim that
\begin{equation}\label{Expansion}
g\ast_{\mathrm{Gutt}} = \sum_{k} \frac{1}{k!} \sum_{i_1, \cdots, i_k}\frac{\partial^k g}{\partial \lambda_{i_1} \cdots \partial \lambda_{i_k}} \bigg |_{\lambda_{i_1}=\cdots=\lambda_{i_k}=0} (\lambda_{i_1}\ast_{\PBW} + \hbar l_{i_1} )\cdots (\lambda_{i_k}\ast_{\PBW} + \hbar l_{i_k}).
\end{equation}	
In fact, for any $f\in C^{\infty}(V\times G)$, we have
\begin{align*}
& g\ast_{\mathrm{Gutt}}f \\
= &\sum_{k}\sum_{i_1,\cdots,i_k}\frac{\hbar^{k}}{k!}\frac{\partial^k g_1}{\partial \lambda_{i_1}\cdots \partial \lambda_{i_k}}\ast_{\PBW}(\overset{\longrightarrow}{l_{i_1}} \cdots \overset{\longrightarrow}{l_{i_k}}f) \\
= &\sum_{k}\sum_{i_1,\cdots,i_k} \frac{\hbar^{k}}{k!} \left(\frac{\partial^k g_1}{\partial \lambda_{i_1}\cdots \partial \lambda_{i_k}} \ast_{\PBW}\right) \cdot (\overset{\longrightarrow}{l_{i_1}} \cdots \overset{\longrightarrow}{l_{i_k}})(f) \qquad\quad  \text{by Equation \eqref{MonoidalPBW}}\\
= &\sum_{k,m}\sum_{i_1,\cdots,i_k}\sum_{j_1,\cdots, j_m} \frac{1}{k!}\frac{1}{m!} \biggl(\frac{\partial^l}{\partial \lambda_{j_1} \cdots \partial\lambda_{j_m} }
\left(\frac{\partial^k g_1}{\partial \lambda_{i_1}\cdots \partial \lambda_{i_k}}\right)
\bigg |_{\lambda_{j_1}=\cdots=\lambda_{j_m}=0} \\
&\qquad\qquad\qquad\qquad\qquad (\lambda_{j_1}\ast_{\PBW}) \cdots (\lambda_{j_m}\ast_{\PBW})(\hbar\overset{\longrightarrow}{l_{i_1}})\cdots
(\hbar\overset{\longrightarrow}{l_{i_k}})\biggr)(f)\\
= &\sum_{k}\frac{1}{k!}\sum_{i_1,\cdots,i_k}
\left(\frac{\partial^k g_1}{\partial \lambda_{i_1}\cdots \partial \lambda_{i_k}}\right)
\bigg |_{\lambda_{i_1}=\cdots=\lambda_{i_k}=0}(\lambda_{i_1}\ast_{\PBW} + \hbar\overset{\longrightarrow}{l_{i_1}} ) \cdots (\lambda_{i_k}\ast_{\PBW} + \hbar\overset{\longrightarrow}{l_{i_k}}) (f).	
\end{align*}
On the other hand, since for all $g \in C^{\infty}(V)$ and linear functions $\lambda_i$ on $\frakl^\ast$,
\begin{equation*}
(\lambda_i\ast_{\PBW})\cdot (g\ast_{\PBW})- (g\ast_{\PBW})\cdot (\lambda_i\ast_{\PBW})
= \hbar ( \ad^{\ast}_{l_i}g),
\end{equation*}
it follows that an element $F\in C^{\infty}(V)\llbracket \hbar \rrbracket\otimes U(\frakg)\otimes U(\frakg)$ is $\frakl$-invariant if and only if
\begin{equation}\label{Eq: Delta commute with Fast}
[\Delta(\lambda_i\ast_{\PBW} \, + \, \hbar l_i)\, ,\overline{F\ast_{\PBW}} ]=0 \in \cH\otimes_{R} \cH.
\end{equation}
Combining \eqref{Expansion} with \eqref{Eq: Delta commute with Fast}, we obtain the desired relation.
\end{proof}

\begin{lemma}\label{Lemma4}
If $F\in C^{\infty}(V)\llbracket \hbar \rrbracket\otimes U(\frakg)\otimes U(\frakg)$ is $\frakl$-invariant, then we have
\begin{equation*}
(\id\otimes_{R} \Delta)\Theta_{\mathrm{Gutt}}\cdot (\overline{F\ast_{\PBW}} )_{23} =
( \overline{\overline{F\ast_{\PBW}}} )_{23} \cdot (\id\otimes_{R} \Delta)\Theta_{\mathrm{Gutt}} \in \cD\otimes_{R} \cH\otimes_{R} \cH,	
\end{equation*}
where $ (\overline{F\ast_{\PBW}} )_{23} = 1 \otimes \overline{F\ast_{\PBW}}$, and
\[
(\overline{ \overline{F\ast_{\PBW}}})_{23}= \sum_{p,r,s} \hbar^p (\id\otimes_{R}\Delta)\Delta(b_{p,rs}\ast_{\PBW}) \cdot (1\otimes U_{p,r} \otimes U_{p,s}).
\]
\end{lemma}
\begin{proof}
We first note that two differential operators $D_1 = D_2\in \cD$ if and only if $D_1 g =D_2 g$ for all polynomial functions $g\in S(\frakl)$. Hence, the calculation below, which holds for all $g_1 \in S(\frakl), g_{2}, g_3 \in C^{\infty}(V\times G)$, verifies the statement:
\begin{align*}
&\quad (\id\otimes_{R} \Delta)\Theta_{\mathrm{Gutt}}\cdot( \overline{F\ast_{\PBW}} )_{23}(g_1\, , g_2\, , g_3) = g_1 \ast_{\mathrm{Gutt}}\biggl(\overline{F\ast_{\PBW}}(g_2\, , g_3) \biggr) \\
&= \left(\Delta(g_1 \ast_{\mathrm{Gutt}})\cdot \overline{F\ast_{\PBW}}\right) (g_2, g_3) \qquad \text{by Lemma \ref{Exchange}}  \\
&=\left( \overline{F\ast_{\PBW}} \cdot\Delta(g_1\ast_{\mathrm{Gutt}})\right) (g_2, g_3)  \\
&= \left(( \overline{\overline{F\ast_{\PBW}}} )_{23} \cdot
(\id\otimes_{R} \Delta)\Theta_{\mathrm{Gutt}}\right)(g_1, g_2, g_3).
\end{align*}  	
\end{proof}

\subsection{Proof of Theorem \ref{MAIN}}\label{Sec:ProofMain}
We  assume  that $F\in C^{\infty}(V)\llbracket \hbar \rrbracket\otimes U(\frakg)\otimes U(\frakg)$ is
$\frakl$-invariant. Then for the element $\cF=\overline{F\ast_{\PBW}}\cdot \Theta_{\mathrm{Gutt}} \in \cH\otimes_{R}\cH$, we have
\begin{align}\label{intermediate4}
&\quad (\Delta\otimes_{R}\id)\cF \cdot (\cF\otimes_{R} 1) \notag \\
& = (\Delta\otimes_{R}\id)\overline{F\ast_{\PBW}} \cdot (\Delta\otimes_{R}\id)\Theta_{\mathrm{Gutt}} \cdot (\overline{F\ast_{\PBW}})_{12}\cdot (\Theta_{\mathrm{Gutt}}\otimes_{R} 1) \quad \text{by Lemma \ref{Lemma2}}\notag \\
& = (\Delta\otimes_{R}\id)\overline{F\ast_{\PBW}} \cdot (\overline{\overline{F\ast_{\PBW}}})_{12}^{(3)} \cdot (\Delta\otimes_{R}\id)\Theta_{\mathrm{Gutt}}\cdot (\Theta_{\mathrm{Gutt}}\otimes_{R} 1),
\end{align}
and
\begin{align}\label{intermediate5}
&\quad (\id\otimes_{R}\Delta)\cF \cdot (1\otimes_{R} \cF) \notag \\
& = (\id\otimes_{R}\Delta)\overline{F\ast_{\PBW}}
\cdot (\id\otimes_{R}\Delta)\Theta_{\mathrm{Gutt}} \cdot(\overline{F\ast_{\PBW}})_{23} \cdot
(1\otimes_{R} \Theta_{\mathrm{Gutt}}) \quad \text{by Lemma \ref{Lemma4}}\notag \\
& = (\id\otimes_{R}\Delta)\overline{F\ast_{\PBW}} \cdot (\overline{\overline{F\ast_{\PBW}}})_{23}
\cdot(\id\otimes_{R}\Delta)\Theta_{\mathrm{Gutt}}\cdot(1\otimes_{R} \Theta_{\mathrm{Gutt}}).
\end{align}
Since $\Theta_{\mathrm{Gutt}}$ is a twistor, we also have
\begin{equation*}
(\Delta\otimes_{R}\id)\Theta_{\mathrm{Gutt}} \cdot (\Theta_{\mathrm{Gutt}}\otimes_{R} 1)=
(\id\otimes_{R} \Delta)\Theta_{\mathrm{Gutt}} \cdot (1\otimes_{R} \Theta_{\mathrm{Gutt}})\, .	
\end{equation*}
Comparing Equations \eqref{intermediate4} and \eqref{intermediate5}, we find that $\cF$ satisfies
\begin{equation*}
(\Delta\otimes_{R}\id)\cF \cdot (\cF\otimes_{R} 1) = (\id\otimes_{R}\Delta)\cF \cdot  (1\otimes_{R} \cF),
\end{equation*}
if and only if
\begin{equation*}
(\Delta\otimes_{R}\id)\overline{F\ast_{\PBW}} \cdot (\overline{\overline{F\ast_{\PBW}}})_{12}^{(3)}
= (\id\otimes_{R}\Delta)\overline{F\ast_{\PBW}} \cdot (\overline{\overline{F\ast_{\PBW}}})_{23},
\end{equation*}
which is indeed equivalent to Equation \eqref{Nonabelian-Twist2}.

Moreover, since $(\epsilon\otimes_{R} \id)\Theta_{\mathrm{Gutt}}=(\id\otimes_{R} \epsilon)\Theta_{\mathrm{Gutt}}=1\otimes_{R} 1$, it follows that $\cF$ satisfies Equation \eqref{counitEquation} if and only if $F$ satisfies Equation \eqref{counit2}.
Therefore, $F$ is a smooth dynamical twist if and only if $\cF=\overline{F\ast_{\PBW}}\cdot \Theta_{\mathrm{Gutt}}$ is a twistor of the quantum groupoid $\cH$.

\section{The Enriquez-Etingof {formal} dynamical twist}\label{Sec:4}
Various constructions of equivariant star products on Poisson homogeneous spaces from dynamical twists have been well-studied (see \cites{Donin-Mudrov-1, Enriquez-Etingof2, Enriquez-Etingof-Marshall, Karolinsky-Muzykin-Stolin-Tarasov, Karolinsky-Stolin-Tarasov-2}).
As an application of Theorem~\ref{MAIN}, we provide an approach to construct such equivariant star products from the quantum groupoid point of view.

Recall that a \textit{polarized Lie algebra} is a finite dimensional $\fK$-Lie algebra $\frakg$ which admits a decomposition
\[
\frakg=\frakl\oplus \fraku=\frakl \oplus \fraku_{+} \oplus \fraku_{-},
\]
where $\fraku_{+}$ and $\fraku_{-}$ are Lie subalgebras of $\frakg$ satisfying
$[\frakl\, , \fraku_{+}]\subset \fraku_{+}$ and $[\frakl\, , \fraku_{-}]\subset \fraku_{-}$.

Each element $\lambda \in \frakl^{\ast}$ determines  a skew-symmetric pairing on $\fraku$:
\[
\omega({\lambda}) \colon \wedge^2 \fraku \rightarrow \fK, \quad u \wedge v \mapsto \lambda([u, v]),
\]
for all $u, v\in \fraku$. By fixing an isomorphism $\fraku^{\ast}\cong \fraku$ of $\fK$-vector spaces,
we may view $\omega(\lambda)$ as an element in $ \operatorname{End}(\fraku)$.

A polarized Lie algebra $\frakg$ is called \textit{nondegenerate} if there exists some $\lambda\in \frakl^{\ast}$ such that $\omega(\lambda)$ is nondegenerate, which is equivalent to that $D(\lambda) :=\det \omega(\lambda) \neq 0$.
Let $\widehat{S}(\frakl)[1/D]$ be the localization of the formal completion of $S(\frakl)$ with respect to $D$.
A deformation quantization of $\widehat{S}(\frakl)[1/D]$ is provided by a completed associative $\fK\llbracket \hbar \rrbracket$-algebra $\widehat{S}(\frakl)[1/D]_{\hbar}$~\cite{Enriquez-Etingof2}. Moreover, it is proved that there exists a unique $\frakl$-invariant element
\begin{equation*}
J \in \widehat{S}(\frakl)[1/D]_{\hbar} \otimes U(\frakg)\otimes U(\frakg),
\end{equation*}
called  the \textbf{Enriquez-Etingof {formal} dynamical twist}, satisfying
\begin{equation*}
(\id\otimes \Delta\otimes \id)J \cdot J_{12}^{(3)} =(\id\otimes \id \otimes \Delta)J  \cdot J_{23},
\end{equation*}
and
\begin{equation*}
(\id\otimes \epsilon\otimes \id)F=(\id\otimes \id\otimes \epsilon)F=1\otimes 1\, ,
\end{equation*}
and certain further conditions \cite{Enriquez-Etingof2}.
(See also~\cites{Donin-Mudrov-1, Donin-Mudrov-2} on construction of dynamical twists via generalized Verma modules and the dynamical adjoint functor.)
By analyzing the construction of $J \in \widehat{S}(\frakl)[1/D]_{\hbar} \otimes U(\frakg)\otimes U(\frakg)$ in \cite{Enriquez-Etingof2}, it is not difficult to see that $J$ indeed comes from a smooth dynamical twist
\begin{equation*}
F_{\frakl}^{\frakg}\in C^{\infty}(V)\llbracket \hbar \rrbracket \otimes U(\frakg)\otimes U(\frakg)
\end{equation*}
over an appropriate open submanifold $V\subset \frakl^{\ast}$.
By Theorem \ref{MAIN}, we obtain a twistor
\begin{equation}\label{Eq: cFgl}
\cF_{\frakl}^{\frakg}:= \overline{F_{\frakl}^{\frakg} \ast_{\PBW}} \cdot \Theta_{\mathrm{Gutt}}
\end{equation}
of the quantum groupoid $\cH = \cD\otimes U(\frakg)\llbracket \hbar \rrbracket$. Here $\cD$ is the algebra of differential operators on $V$.

Assume that the closed submanifold $W \subset V$ consisting of characters of 1-dimensional trivial representations of $\frakl$ is nonempty.
Let $G $ be a Lie group integrating $\frakg$ and $L \subset G$ be a Lie subgroup integrating $\frakl \subset \frakg$.
Consider the space of $W$-parameterized family of smooth differential operators on the homogeneous space $G/L$
\begin{equation*}
\mathcal{K}:= C^{\infty}(W) \otimes U(T(G/L)) \llbracket \hbar \rrbracket,	
\end{equation*}
 which is indeed a quantum groupoid over $S:=C^{\infty}(W) \llbracket \hbar \rrbracket$.
\begin{proposition}
The twistor $\cF_{\frakl}^{\frakg}$~\eqref{Eq: cFgl} of the quantum groupoid $\cH$ induces a twistor
\begin{equation*}
\cF_{\frakg/\frakl} \in \cK \otimes_{S} \cK
\end{equation*}
of the quantum groupoid $\cK$, which gives rise to a $W$-parameterized family of left $G$-invariant star products on $G/L$, being a quantization of the $W$-parameterized family of Kirillov-Kostant-Souriau Poisson structures on $G/L$.
\end{proposition}
\begin{proof}
Note that the twistor $\cF_{\frakl}^{\frakg}$ induces a right $L$-invariant compatible star product $\ast$ on $V \times G$ defined by
\[
f_1\ast f_2 =\cF_{\frakl}^{\frakg}(f_1, f_2),
\]
for all $f_1, f_2\in C^{\infty}(V\times G)$.
Given $g_1, g_2\in C^{\infty}(G/L)$, viewed as right $L$-invariant functions on $V\times G$, their star product $g_1\ast g_2$ is also right $L$-invariant.
Thus, the restriction $(g_1\ast g_2) |_{W\times G}$ descends to an element in $C^{\infty}(W)\otimes C^{\infty}(G/L)\llbracket \hbar \rrbracket$, which can be viewed as a family of products $g_1\star_{\lambda} g_2$ on $C^\infty(G/L)\llbracket \hbar \rrbracket$ parametrized by $\lambda \in W$.
Note that the coadjoint orbit $O_{\lambda}\subset \frakg^{\ast}$ through $\lambda \in W$ is diffeomorphic to $G/L$. Each $\star_{\lambda}$ is a left $G$-invariant star product on $C^{\infty}(G/L)\llbracket \hbar \rrbracket$, whose classical limit is the Kirillov-Kostant-Souriau Poisson structure on $O_{\lambda} \cong G/L$~\cite{Enriquez-Etingof2}. This family $\star_{\lambda}$ of star products corresponds
to a twistor $\cF_{\frakg/\frakl}\in \cK\otimes_{S}\cK$ of the quantum groupoid $\cK$.
\end{proof}

\begin{remark}
Assume that a Lie algebra $\frakg$ admits a decomposition $\frakg = \frakl \oplus \fraku$,
where $\frakl$ is abelian and $\fraku$ is a Lie subalgebra.
It is proved in ~\cite{Karolinsky-Stolin-Tarasov-1} that non-dynamical twists can be produced from dynamical twists. Our construction of the twistor $\cF_{\frakg/\frakl}$ can be viewed as a generalization of this procedure to the case of a Lie algebra pair $(\frakg\, , \frakl)$ arising from the nondegenerate polarized Lie algebra $\frakg$.
\end{remark}

\end{document}